\documentclass[11pt,letterpaper]{amsart}

\usepackage[
colorlinks=true,              
linkcolor=blue,                
citecolor=red,               
urlcolor=cyan                  
]{hyperref}

\numberwithin{equation}{section}
\theoremstyle{plain}
\newtheorem{theorem}{Theorem}[section]
\newtheorem{prop}[theorem]{Proposition}

\newtheorem{lemma}[theorem]{Lemma}

\newtheorem{corollary}[theorem]{Corollary}

\theoremstyle{definition}

\newtheorem{example}[theorem]{Example}

\newtheorem{remark}[theorem]{Remark}

\newcommand{\Rmnum}[1]{\expandafter\@slowromancap\romannumeral #1@}

\newcommand{\Ric}{\operatorname{Ric}}
\newcommand{\mr}{\mathbb{R}}
\newcommand{\ud}{\mathrm{d}}
\newcommand{\Hess}{\mathrm{Hess}}
\newcommand{\ms}{\mathbb{S}}

\newcommand{\diam}{\mathrm{diam}}

\allowdisplaybreaks


\author{Xumin Jiang}
\address{Xumin Jiang, School of Sciences, Great Bay University, Dongguan 523000, China}
\email{xjiang@gbu.edu.cn}

\author{Mingxiang Li}
\address{Mingxiang Li, Department  of Mathematics \& Institue of Mathematical Sciences, The Chinese University of Hong Kong}
\email{limx@smail.nju.edu.cn}

\author{Zhehui Wang}
\address{Zhehui Wang, School of Sciences, Great Bay University, Dongguan 523000, China}
\email{wangzhehui@gbu.edu.cn}

\begin{document}
	\title{Bonnet-Myers  type theorems for 
		$Q$-curvature on four-manifolds}
	
	\date{}
	\maketitle
	
	\begin{abstract}
		Let $(M^4,g)$ be a complete four-dimensional Riemannian manifold. First,  if the $Q$-curvature  $Q_g\geq 6k^2$ and scalar curvature $R_g\geq -12k$ for some positive constant $k$,
         then $(M^4,g)$ is   either Einstein with $\Ric_g=-3kg$  or  compact with $R_g\ge 12k$.   As a corollary,  the fundamental group  \(\pi_1(M^4)\)  satisfies $|\pi_1(M^4)|\leq 16\pi^2/\left(\int_{M^4}Q_g\ud\mu_g\right), $ under the additional assumption $R_g>-12k$.
          Second, if the scalar curvature $R_g>0$ and  $Q_g\geq \theta R_g$ for  a positive constant $\theta$, 
              then  $M^4$ is compact and 
 the diameter of $(M^4,g)$ is at most $4\pi/\sqrt{15\theta}$. 
	\end{abstract}
	
	\section{Introduction}

 On a complete four-manifold $(M^4,g)$, Branson's $Q$-curvature $Q_g$ is given by the following  formula
	\begin{equation}\label{def of Q}
		Q_g=\frac{1}{6}\left(-\Delta_g R_g-3|\Ric_g|_g^2+R_g^2\right)
	\end{equation}
	where $R_g$ denotes the scalar curvature, $\Ric_g$ the Ricci tensor,  and $\Delta_g$ the Laplace–Beltrami operator. Our normalization of $Q$-curvature differs by a factor of two from the convention used in
some references.  Compared with other curvature quantities, the 
	$Q$-curvature is a relatively young notion and remains less thoroughly understood.
	
	 A crucial property of the 
	$Q$-curvature on compact 4-manifolds is the conformal invariance of its integral, which naturally leads one to regard it as a generalization of the Gaussian curvature for compact surfaces. 
Using $Q$-curvature, we can rewrite the four-dimensional Chern-Gauss-Bonnet formula as follows
	\begin{equation}\label{eq:CGB-formula}
    \int_{M^4}Q_g\ud\mu_g+\frac{1}{4}\int_{M^4}|W_g|_g^2\ud\mu_g=8\pi^2\chi(M^4)
	\end{equation}
	where $W_g$ and $\chi(M^4)$ denote Weyl tensor and Euler characteristic respectively. Together with the Paneitz operator, the 
	$Q$-curvature has become one of the central topics in conformal geometry. We refer to the notable works including Chang–Yang \cite{Chang-Yang}, Djadli-Malchiodi \cite{DM}, and Gursky \cite{Gursky Ann,GUrsky CMP 97, Gursky CMP} for further discussions on 	$Q$-curvature on compact four-manifolds. 
    
    For any conformal metric $\tilde g\in [g]$ of the compact 4-manifold $(M^4,g)$, the following two integral quantities are conformal invariant,
\begin{equation}\label{eq:Q-conformal-invar}
    \int_{M^4}Q_{\tilde g}\ud\mu_{\tilde g}=\int_{M^4}Q_g\ud\mu_g
\end{equation}
and 
\begin{equation}\label{eq:Weyl-conformal-inva}
    \int_{M^4}|W_{\tilde g}|_{\tilde g}^2\ud\mu_{\tilde g}=    \int_{M^4}|W_g|_g^2\ud\mu_g.
\end{equation}
The four-dimensional Chern–Gauss–Bonnet formula \eqref{eq:CGB-formula} reveals a profound connection between the two conformal invariants \eqref{eq:Q-conformal-invar} and \eqref{eq:Weyl-conformal-inva}, which together govern the topology of compact 4-manifolds.
The Yamabe invariant, another conformal invariant quantity defined by
\begin{equation}\label{eq:Yamabe-invariant}
    Y(M^4,[g]):=\inf_{\tilde g=e^{2w}g}\frac{\int_{M^4}R_{\tilde g}\ud\mu_{\tilde g}}{(\int_{M^4}\ud\mu_{\tilde g})^{1/2}},
\end{equation}
is well known to be strongly tied to the topology of compact manifolds.
The three conformal invariants \eqref{eq:Q-conformal-invar}, \eqref{eq:Weyl-conformal-inva}, and \eqref{eq:Yamabe-invariant} have a surprising influence on compact four-manifolds. Let us recall  these remarkable contributions.

For a sufficiently small positive constant  $\varepsilon_0$, if  the two conformal invariant quantities \eqref{eq:Weyl-conformal-inva} and  \eqref{eq:Yamabe-invariant} satisfy
\begin{equation*}\label{Gursky-IUMJ-condition}
    \int_{M^4}|W_g|_g^2\ud\mu_g<\varepsilon_0,\qquad  Y(M^4,[g])\geq 0, 
\end{equation*}
Gursky \cite{Gursky-IUMJ} showed that  the Euler characteristic is controlled by 
\begin{equation*}
    \chi(M^4)\leq 2.
\end{equation*}
Later, in another  remarkable work \cite[Corollary F]{Gursky Ann}, he  showed that  under the assumptions 
\begin{equation}\label{eq:positive-Q-and-Y}
    \int_{M^4}Q_g\ud\mu_g >0, \qquad Y(M^4,[g])>0,
\end{equation}
the first Betti number $b_1(M^4)$ satisfies 
\begin{equation}\label{eq:b_1=0}
       b_1(M^4)=0.
\end{equation}
Under the condition
 \begin{equation*}
     Y(M^4,[g])\geq 0,
 \end{equation*}
 another interesting and important result, due to Gursky \cite{Gursky CMP}, is the following upper bound:
 \begin{equation}\label{Gursky upper bound}
 	\int_{M^4}Q_g\ud\mu_g\leq 16\pi^2=6|\mathbb{S}^4|,
 \end{equation} 
 where $|\mathbb{S}^4|$ denotes the volume of standard four-sphere  and the  equality holds if and only if $(M^4,g)$ is  conformally equivalent  to the standard four-sphere.
 In the pioneering works of Chang, Gursky, and Yang \cite{CGY,CGY JAM} (see also \cite{GV}), under the same condition \eqref{eq:positive-Q-and-Y}, they show that 
\begin{equation}\label{eq:conformal-postive-Ricci}
    \exists \tilde g\in [g], \quad \Ric_{\tilde g}>0.
\end{equation}
Once the  positivity of  Ricci curvature \eqref{eq:conformal-postive-Ricci} is obtained, a wealth of geometric and topological information about the manifold follows. For example, by the Bochner technique, the conclusion \eqref{eq:b_1=0}
 follows immediately. 
 Moreover, combining \eqref{eq:conformal-postive-Ricci} with the Bonnet--Myers theorem implies that the fundamental group $\pi_1(M^4)$ is finite. Furthermore, under the assumption \eqref{eq:positive-Q-and-Y}, the  Chang-Gursky-Yang result \eqref{eq:conformal-postive-Ricci} together with Gursky's inequality \eqref{Gursky upper bound} yields that  
\begin{equation}\label{pi_1-equation}
    |\pi_1(M^4)| \leq \frac{16\pi^2}{\int_{M^4} Q_g \ud\mu_g}.
\end{equation}
Although the above inequality follows as a corollary of the results in \cite{CGY} and \eqref{Gursky upper bound}, it appears not to have been explicitly stated elsewhere. We therefore include it as Proposition \ref{pro:p_1-formula} and provide a brief proof there.

 Moreover, based on their work \cite{CGY, CGY JAM}, Chang, Gursky, and Yang \cite{CGY IHES} (see also \cite{CGZ}) established a beautiful  conformally invariant sphere theorem for compact four-manifolds: under a sharp integral pinching condition,
 \begin{equation*}\label{eq:pinching-condition}
     \int_{M^4}Q_g\ud\mu_g>\frac{1}{4}\int_{M^4}|W_g|^2\ud\mu_g, \qquad Y(M^4,[g])>0,
 \end{equation*}
 then $M^4$ is diffeomorphic to either $\ms^4$ or $\mathbb{RP}^4$.

These remarkable results lead us to believe that the  $Q$-curvature may have a much deeper connection to the manifold's topology than we had expected. As we noted before, under the conformal invariant condition \eqref{eq:positive-Q-and-Y}, Chang, Gursky and Yang \cite{CGY}  established \eqref{eq:conformal-postive-Ricci}. Under the same assumption \eqref{eq:positive-Q-and-Y}, Lee \cite{Lee} showed that there exists a conformal metric $\tilde g\in[g]$ such that 
 \begin{equation}\label{postiveQandR}
 	Q_{\tilde g}>0, \quad R_{\tilde g}>0, \quad \forall x\in M^4.
 \end{equation}
Consequently, the assumption \eqref{eq:positive-Q-and-Y} is equivalent to the existence of a conformal metric $\tilde g\in[g]$ satisfying the pointwise inequalities \eqref{postiveQandR}.

These results suggest that, on four-manifolds, the positivity of both the 
$Q$-curvature and the scalar curvature imposes stronger constraints than does the positivity of the Ricci curvature. All of the above results assume compactness of the manifold. A natural question then arises: {\bf can the compactness of the manifold be obtained by imposing suitable conditions on the 
$Q$-curvature and the scalar curvature?}

Before dealing  with such a question,  let us first  recall the classical Bonnet-Myers theorem. For a complete Riemannian manifold $(M^n,g)$, the Bonnet-Myers theorem  \cite{Myers} asserts that   if the  Ricci curvature satisfies \begin{align*}
\Ric_g\geq (n-1)kg
\end{align*} for some positive constant $k$, then $M^n$ is compact and its  diameter is controlled by 
    \begin{equation}\label{eq:diam-Bonnet-Myers}
       \diam(M^n,g) \leq  \pi/\sqrt{k}.
    \end{equation} Moreover, Cheng's rigidity theorem \cite{Cheng} shows that  if  equality in  \eqref{eq:diam-Bonnet-Myers} holds,  then $(M^n,g)$ is isometric to a round sphere. 
	 These elegant and foundational results in Riemannian geometry beautifully link curvature lower bounds to global topological properties. In addition, there have been numerous generalizations of Bonnet–Myers type theorems; see, for instance, \cite{Ambrose, BHJ, Calabi, CHY, CGT, X. Li, ShenYe Duke, ShenYe} and the references therein. 

In fact, about two decades ago, for a standard 4-sphere $(\ms^4,g_0)$, Chang, Hang and Yang \cite{CHY} considered an open connected domain $\Omega\subset\ms^4$ with a complete conformal metric $g=e^{2u}g_0$. Under the conditions
\begin{equation*}
    |R_g|+|\nabla^gR_g|_g\leq c_0, \quad R_g\geq c_1>0,\quad  Q_g\geq c_2>0
\end{equation*}
for some positive constants $c_0,c_1,c_2$, they showed that $\Omega=\ms^4$. By combining \cite[Theorem 1.1]{LWX} with the second author's work \cite[Theorem 1.4]{Li 26 RMI}, one can directly  obtain that there exists no complete conformal metric $g=e^{2u}|\ud x|^2$ on $\mr^4$ such that 
$$Q_g\geq 1, \qquad R_g\geq 0.$$
These results, to some extent, provide an answer to our previous question in the  particular setting.

In this paper, we aim to establish a Bonnet–Myers type theorem in terms of the 
$Q$-curvature and to answer the question raised above. Our first theorem is the following.

    \begin{theorem}\label{thm:Bonnet myers}
Let $(M^4,g)$ be a smooth, complete and connected four-dimensional Riemannian manifold. Suppose the $Q$-curvature $Q_g$ and scalar curvature $R_g$ satisfy
\begin{align}\label{condition-of-Q-R}
    Q_g \geq 6k^2 \quad \text{and} \quad R_g \geq -12k
\end{align}
for some constant $k>0$.
Then either $(M^4,g)$ is Einstein with $\Ric_g = -3kg$, or $M^4$ is compact with $R_g \geq 12k$.
\end{theorem}
\begin{remark}
     The standard hyperbolic space $(\mathbb{H}^4,g_0)$ satisfies
     \begin{align}
         Q_{g_0}=6, \quad R_{g_0}=-12
     \end{align}
     and the  standard product manifold $(\mr\times\ms^3,g_1)$ satisfies \begin{align}
         Q_{g_1}=0,  \quad R_{g_1}=6.
     \end{align} Up to a scaling, these two noncompact models show that the assumptions \eqref{condition-of-Q-R}  and $k>0$ are sharp. 
\end{remark}

As a direct application of Theorem \ref{thm:Bonnet myers},  we obtain the control of the fundamental group via the total  $Q$-curvature.
\begin{corollary}\label{cor:pi_1-control-for-Q}
     Let $(M^4,g)$ be a smooth, complete and connected four-dimensional Riemannian manifold. Suppose that, after an appropriate scaling of $g,$ the $Q$-curvature $Q_g$ and scalar curvature $R_g$ satisfy
   $$Q_g\geq 6\quad \mathrm{and} \quad R_g>-12.$$
    Then the fundamental group $\pi_1(M^4)$ is finite with \eqref{pi_1-equation} holds.
\end{corollary}
  
Unlike the classical Bonnet-Myers theorem, the assumption \eqref{condition-of-Q-R} does not allow one to control the diameter of the compact manifold $M^4$ (excluding the negative Einstein case); see Example \ref{example}.
 Interestingly, we find that if the quotient $Q_g/R_g$ admits a positive lower bound, then we can control the diameter. This quotient is itself an interesting new geometric quantity, and Ge-Wang-Wei \cite{GWW} investigated a Yamabe type theorem involving $Q_g/R_g$. Our second Bonnet-Myers type result is as follows.

\begin{theorem}\label{thm:diam}
Let $(M^4,g)$ be a smooth, complete, connected four‑dimensional Riemannian manifold. Suppose  the $Q$-curvature $Q_g$ and scalar curvature $R_g$ satisfy
\begin{align}\label{condition-of-diam}
R_g\geq 0 \quad \text{and} \quad Q_g\geq \theta R_g.
\end{align}
for some constant $\theta>0$.
Then  either $(M^4,g)$ is  Ricci‑flat, or $M^4$ is compact with $R_g\geq 24\theta$ and  the diameter admits an upper bound
\begin{equation}\label{eq:diam-control}
\operatorname{diam}(M^4,g)\leq 4\pi /\sqrt{15\theta}.
\end{equation}
\end{theorem}

\begin{remark}
We do not expect the  upper bound \eqref{eq:diam-control} is sharp.  It would be interesting to determine the sharp upper bound and the corresponding rigidity behavior.
\end{remark}

	To conclude the introduction, we briefly outline the structure of this paper. In Section \ref{sec:key lemmas}, we prove  a modified Chang-Gursky-Yang's lemma established in   \cite{CGY} and 	recall a property of the conformal Ricci tensor introduced by Shen–Ye \cite{ShenYe}. Meanwhile, we  establish a lemma which is essentially the one-dimensional Allegretto–Piepenbrink theorem. In Section \ref{sec:positive-Q-and-R}, under a more restrictive condition compared with \eqref{condition-of-Q-R}, we show the  complete manifold must be compact (see Theorem \ref{thm:Bonnet myers-2}). Meanwhile, we sketch the proof of \eqref{pi_1-equation} and prove Corollary \ref{cor:pi_1-control-for-Q}.  In Section \ref{sec:compelte-metric}, we show that if $R_g>-12k$, the conformal metric $(R_g+12k)g$ must be complete (see Lemma \ref{lem:g_S-complete}). Building on this, we find  a gap   phenomenon for $R_g$ (see Proposition \ref{prop:gap-R}). Then, in Section \ref{sec:diameter-control}, we finish the proof of Theorems \ref{thm:Bonnet myers} and \ref{thm:diam}.
In the final Section \ref{sec:discussions}, we discuss some connections between  $Q$-curvature and  other geometric quantities. Meanwhile,  we  provide an example concerning the diameter under the condition \eqref{condition-of-Q-R}.

\vspace{1em}
		
	{\bf Acknowledgment.} The second author would like to   thank Jintian Zhu for helpful discussions.

	\section{Key lemmas}\label{sec:key lemmas}
	
	As noted  in the introduction, Chang-Gursky-Yang \cite{CGY} discovered that  under \eqref{eq:positive-Q-and-Y}, there exists   a conformal metric with positive Ricci curvature. In their work, they considered  the $k$-th elementary symmetric functions $\sigma_k(A_g)$ of the Schouten tensor $A_g$. For an $n$-dimensional manifold $(M^n,g)$, the Schouten tensor  $A_g$ is defined by 
    \begin{align*}
    A_g:=\frac{1}{n-2}\left(\Ric_g-\frac{R_g}{2(n-1)}g\right).
    \end{align*}
	In particular, in dimension four, a direct computation yields
	\begin{align*}
    \sigma_1(A_g)=\frac{R_g}{6}, \quad \sigma_2(A_g)=\frac{1}{24}\left(-3|\Ric_g|^2_g+R_g^2\right).
    \end{align*}
	Using such identities together with \eqref{def of Q}, we can rewrite $Q_g$ as
	\begin{align*}
    Q_g=-\Delta_g\sigma_1(A_g)+4\sigma_2(A_g).
    \end{align*}
	When $(M^4,g)$ is compact, we have the integral identity
	\begin{align*}
    \int_{M^4} Q_g\ud\mu_g=4\int_{M^4}\sigma_2(A_g)\ud\mu_g.
    \end{align*}
    
Under the assumption \eqref{eq:positive-Q-and-Y},
	Chang, Gursky, and Yang \cite{CGY} employ the continuity method to find a conformal metric $\tilde g=e^{2u}g$    such that 
    \begin{equation}\label{eq:postive-sgima-1-2}
        \sigma_2(A_{\tilde g})>0
	,\quad \sigma_1(A_{\tilde g})>0.
    \end{equation}
    Then, using the  following crucial observation that, for a 4-manifold $(M^4,  g)$ with
    \begin{align*}
        R_{g}>0,
    \end{align*}  the Ricci curvature satisfies the inequality
	\begin{equation}\label{CGY lemma}
		\Ric_{ g}\geq \frac{12\sigma_2(A_{ g})}{R_{ g}} g.
	\end{equation}
Under \eqref{eq:postive-sgima-1-2}, it follows immediately that 
\begin{align*}
    \Ric_{ \tilde g}>0.
\end{align*}
 Note that
\eqref{CGY lemma} is precisely Lemma 1.2 in \cite{CGY}. 

In this paper, we replace $\sigma_2(A_g)$ by $Q_g$ in \eqref{CGY lemma}. This replacement naturally connects $Q$-curvature  to the conformal Ricci tensor introduced by Shen–Ye \cite{ShenYe}, which will be discussed in more detail later.  Different from Chang-Gursky-Yang's inequality \eqref{CGY lemma}, we modify it to fit  our setting.
	\begin{lemma}\label{lem:modified CGYlemma}
		Let $p\in M^4$ and each unit vector $v\in T_pM^4$. If there exists a constant $\varepsilon$ such that  $R_g+\varepsilon$ is positive  at $p$, then there holds
		\begin{equation}\label{eq:shifted-vector}
			\Ric_g(v,v)-\frac{\Delta_g(R_g+\varepsilon)}{2(R_g+\varepsilon)}
			\ge
			\frac{3Q_g-\varepsilon^2/8}{R_g+\varepsilon}.
		\end{equation}
	\end{lemma}
	\begin{proof}
		Choose an orthonormal frame 
        \begin{align*}
           e_1=v,\; e_2,\; e_3,\; e_4 
        \end{align*} 
		and denote $\Ric_g(e_i,e_i)$ by $\lambda_i$. Using Cauchy's inequality, one has 
			\begin{align*}
            \begin{split}
			|\Ric_g|_g^2 &= \lambda_1^2+\lambda_2^2+\lambda_3^2+\lambda_4^2\\
			&\geq \lambda_1^2+\frac{\left(\lambda_2+\lambda_3+\lambda_4\right)^2}{3}\\
			&=\lambda_1^2+\frac{(R_g-\lambda_1)^2}{3}\\
			&= \frac{4}{3}\lambda_1^2-\frac{2}{3}\lambda_1R_g+\frac{R_g^2}{3}.
            \end{split}
		\end{align*}
		Combining with the definition \eqref{def of Q}, there holds
		\begin{align*}
			\begin{split}&2(R_g+\varepsilon)
			\left(\Ric_g(v,v)-\frac{\Delta_gR_g}{2(R_g+\varepsilon)}\right)-6Q_g\\
			&\qquad=2(R_g+\varepsilon)\lambda_1+3|\Ric_g|_g^2-R_g^2\\
			&\qquad\ge4\lambda_1^2+2\varepsilon\lambda_1\\
			&\qquad=4\left(\lambda_1+\frac{\varepsilon}{4}\right)^2
			-\frac{\varepsilon^2}{4}
			\ge-\frac{\varepsilon^2}{4}.
            \end{split}
		\end{align*}
		After division by $2(R_g+\varepsilon)>0$ and using the fact $\Delta_gR_g=\Delta_g(R_g+\varepsilon)$,  we obtain our desired estimate
		\eqref{eq:shifted-vector}.
	\end{proof}

	For a positive function $f$ and a parameter $\sigma>0$, 
	Shen and Ye \cite{ShenYe} introduced the following  conformal Ricci tensor:
	\begin{equation}\label{eq:confRicdef}
		\Ric^{f,\sigma}_g
		:=\Ric_g-\sigma f^{-1}(\Delta_g f)g.
	\end{equation}
If $R_g+\varepsilon>0$, 	choosing 
\begin{align*}
    f=R_g+\varepsilon, \quad \sigma=\frac12,
\end{align*} Lemma \ref{lem:modified CGYlemma} yields that 
	\begin{equation}\label{eq:conf-Ric-R}
			\Ric^{R_g+\varepsilon,\frac{1}{2}}_g\ge\frac{3Q_g-\varepsilon^2/8}{R_g+\varepsilon}g.
	\end{equation}
    
We will use the following second-variation inequality for minimizing geodesics involving the conformal Ricci tensor, established by Shen and Ye in Lemma 1 of \cite{ShenYe}. For the reader’s convenience, we include their proof here.
	
	\begin{lemma}\label{lem:SY}
		Let $(N^n,h)$ be a Riemannian manifold with $n\geq 2$. Let
		$f\in C^\infty(N)$ satisfy $f>0$, let $\sigma\not=0$, and set
		\begin{align}
		\widetilde h=f^{2\sigma}h.
		\end{align}
		Suppose that a $\widetilde h$-geodesic $\widetilde\gamma$ is
		reparametrized by $h$-arc length $s$, denoted by $\gamma(s)$, and is
		minimizing to second order under fixed-endpoint variations. Then, for
		every smooth function $\phi$ vanishing at the endpoints, one has
		\begin{align*}
        \begin{split}
			&(n-1)\int_0^L (\phi')^2\ud s
			+\frac{n-1}{4}\int_0^L \phi^2
			\left(\frac{\ud}{\ud s}\log (f^\sigma)\right)^2\ud s
			\\
			&\qquad
			+(3-n)\int_0^L \phi\phi'
			\frac{\ud}{\ud s}\log (f^\sigma) \ud s
			\\
			&\qquad\geq
			\int_0^L \phi^2
			\Ric_h^{f,\sigma}(\dot\gamma,\dot\gamma)\ud s
			+\frac{1}{\sigma}\int_0^L \phi^2
			\left|\nabla\log (f^\sigma)\right|_h^2 \ud s.
			\label{eq:SYgeneral}
            \end{split}
		\end{align*}
	\end{lemma}
	
	\begin{proof}
		For brevity, set 
        \begin{align*}
           u:=\log (f^\sigma)=\sigma\log f. 
        \end{align*}
		Suppose that
		\begin{align*}
		\gamma:[0,L]\longrightarrow N
		\end{align*}
		is the $h$-arc length reparametrization of $\widetilde\gamma$, and write
		\begin{align*}
		T:=\dot\gamma,\quad
		|T|_h=1.
		\end{align*}
		If $\widetilde s$ denotes the $\widetilde h$-arc length parameter, then
		\begin{align*}
		\ud\widetilde s=e^u\ud s, \quad
		\widetilde T
		:=\frac{\ud\widetilde\gamma}{\ud\widetilde s}
		=e^{-u}T.
		\end{align*}
			Recall that, under the conformal change 
		\begin{align*}
		    \widetilde h=e^{2u}h,
            \end{align*}
		the Levi-Civita connections are related by (See page 58 in \cite{Besse})
		\begin{equation}\label{conformal connection}
				\widetilde\nabla_XY
			=
			\nabla_XY
			+\ud u(X)Y
			+\ud u(Y)X
			-h(X,Y)\nabla u.
		\end{equation}
		Since $\widetilde\gamma$ is a $\widetilde h$-geodesic, one has 
		\begin{align*}
		0
		=
		\widetilde\nabla_{\widetilde T}\widetilde T
		=
		e^{-2u}
		\left(\widetilde\nabla_TT-u'T\right),
		\end{align*}
		where
		\begin{align*}
        u'	=\frac{\ud}{\ud s}(u\circ\gamma).	
        \end{align*}
		Thus, we obtain
		\begin{align}
		\widetilde\nabla_TT=u'T.
		\end{align}
		On the other hand, using \eqref{conformal connection}, one has
		\begin{align*}
		\nabla_TT
		=
		\nabla u-u'T
		=
		(\nabla u)^\perp
		\end{align*}
		where $\perp$ denotes the projection to the orthogonal complement of $T$.
		It follows that
		\begin{equation}	\label{eq:hess-u-gamma}
			\Hess_hu(T,T)=	\frac{\ud^2}{\ud s^2}(u\circ\gamma)
			-\ud u(\nabla_TT)=	u''
			-\left|(\nabla u)^\perp\right|_h^2.
		\end{equation}
			The conformal transformation formula (See page 58 in \cite{Besse}) for the Ricci tensor is
		\begin{align*}\begin{split}
		\Ric_{\widetilde h}
		&=
		\Ric_h
		-(n-2)\bigl(\Hess_hu-\ud u\otimes\ud u\bigr)\\
		&\qquad-\bigl(\Delta_hu+(n-2)|\nabla u|_h^2\bigr)h.\end{split}
		\end{align*}
		Evaluating this identity on $T$, we obtain
		\begin{align*}
			\begin{split}\Ric_{\widetilde h}(T,T)
			&=
			\Ric_h(T,T)-\Delta_hu
			-(n-2)\Hess_hu(T,T)\\
			&\qquad
			+(n-2)(u')^2
			-(n-2)|\nabla u|_h^2.\end{split}
		\end{align*}
	With the help of 
    \begin{align*}
		|\nabla u|_h^2
		=
		(u')^2+\left|(\nabla u)^\perp\right|_h^2,
		\end{align*} and using \eqref{eq:hess-u-gamma}, the gradient terms cancel and hence
		\begin{equation*}
			\Ric_{\widetilde h}(T,T)
			=
			\Ric_h(T,T)-\Delta_hu-(n-2)u''.
			\label{eq:conformal-Ricci-along-gamma}
		\end{equation*}
		By the definition \eqref{eq:confRicdef}, 
		we have
		\begin{equation}\label{RicTT}
			\Ric_{\widetilde h}(T,T)
			=
			\Ric_h^{f,\sigma}(T,T)
			-(n-2)u''
			+\frac{1}{\sigma}|\nabla u|_h^2.
		\end{equation}
		
		We next use the second variation inequality, for every
		smooth function $\psi$ vanishing at the endpoints,  one has
		\begin{align*}
		(n-1)\int_0^{\widetilde L}
		\left(\frac{\ud\psi}{\ud\widetilde s}\right)^2
		\ud\widetilde s
		\geq
		\int_0^{\widetilde L}
		\psi^2
		\Ric_{\widetilde h}(\widetilde T,\widetilde T)
		\ud\widetilde s.
		\end{align*}
		By changing the variable, 
		this becomes
		\begin{align*}
		(n-1)\int_0^L e^{-u}(\psi')^2\ud s
		\geq
		\int_0^L e^{-u}\psi^2
		\Ric_{\widetilde h}(T,T)\ud s.
		\end{align*}
	Inserting \eqref{RicTT} into the above inequality, one has
		\begin{align*}\begin{split}
			(n-1)\int_0^L e^{-u}(\psi')^2\ud s
			\geq{}&
			\int_0^L e^{-u}\psi^2
			\Ric_h^{f,\sigma}(T,T)\ud s
			\\
			&-(n-2)\int_0^L e^{-u}\psi^2u''\ud s
			\\
			&+\frac{1}{\sigma}\int_0^L e^{-u}\psi^2
			|\nabla u|_h^2\ud s.
			\label{eq:index-u}\end{split}
		\end{align*}
		
		Now let   $\psi=e^{u/2}\phi$
		where $\phi$ is smooth and vanishes at the endpoints. Then,  we obtain
		\begin{align*}\begin{split}
			&(n-1)\int_0^L(\phi')^2\ud s
			+(n-1)\int_0^L u'\phi\phi'\ud s
			+\frac{n-1}{4}\int_0^L(u')^2\phi^2\ud s\\
			&\geq
			\int_0^L\phi^2
			\Ric_h^{f,\sigma}(T,T)\ud s
			-(n-2)\int_0^L\phi^2u''\ud s
			+\frac{1}{\sigma}\int_0^L\phi^2
			|\nabla u|_h^2\ud s.\end{split}
		\end{align*}
		Since $\phi$ vanishes at the endpoints, integration by parts gives
		\begin{align*}
		\int_0^L\phi^2u''\ud s
		=
		-2\int_0^L\phi\phi'u'\ud s.
		\end{align*}
		Therefore, combining these two estimates 
		\begin{align*}\begin{split}
			&(n-1)\int_0^L(\phi')^2\ud s
			+\frac{n-1}{4}\int_0^L\phi^2(u')^2\ud s
			+(3-n)
			\int_0^L\phi\phi'u'\ud s\\
			&\geq
			\int_0^L\phi^2
			\Ric_h^{f,\sigma}(T,T)\ud s
			+\frac{1}{\sigma}\int_0^L\phi^2
			|\nabla u|_h^2\ud s\end{split}
		\end{align*}
		which is our desired estimate.
	\end{proof}
	
Another ingredient is the following lemma  which is essentially one-dimensional  Allegretto–Piepenbrink theorem (see \cite{Simon, Pin} for more details). For the reader’s convenience, we sketch the proof here.

	\begin{lemma}\label{lem:positive-solution}
		Let \(V\in C^\infty((0,\infty))\) satisfying \(V\geq 0\). Assume that
		\begin{equation}\label{eq:nonnegative-operator}
			\int_0^\infty |\phi'(s)|^2\ud s
			-
			\int_0^\infty V(s)\phi(s)^2\ud s
			\geq 0,
			\qquad
			\forall\,\phi\in C_c^\infty((0,\infty)).
		\end{equation}
		Then there exists a strictly positive function
		\(y\in C^\infty((0,\infty))\) satisfying
		\begin{equation}\label{eq:positive-y}
			y''+Vy=0
			\qquad\text{on }(0,\infty).
		\end{equation}
	\end{lemma}
	
	\begin{proof}
		For every integer \(j\geq 2\), set the interval 
		\begin{align*}
		I_j:=\left(\frac1j,j\right)
		\end{align*}
		and define
		\begin{align*}
		\mu_j
		:=
		\inf_{\substack{\phi\in H_0^1(I_j)\\ \phi\not\equiv 0}}
		\frac{
			\displaystyle
			\int_{I_j}
			\left(|\phi'|^2-V\phi^2\right)\ud x
		}{
			\displaystyle
			\int_{I_j}\phi^2\ud x
		}.
		\end{align*}
			Since $\overline{I_j}\subset(0,\infty)$,  the function \(V\) is bounded
		on \(I_j\). Moreover, for every $\phi\in C_c^\infty(I_j)$,  assumption
		\eqref{eq:nonnegative-operator} gives
		\begin{align*}
		\int_{I_j}
		\left(|\phi'|^2-V\phi^2\right)\ud x
		\geq 0.
		\end{align*}
		 Consequently,  one has \begin{align}	\mu_j\geq 0.\end{align}
		By the direct method of the calculus of variations, the infimum
		\(\mu_j\) is attained by some nonnegative  function
        \begin{align*}
        \phi_j\in H_0^1(I_j).
        \end{align*}
        We normalize
		\(\phi_j\) so that
		\begin{align*}
		\|\phi_j\|_{L^2(I_j)}=1.
		\end{align*}
		The Euler--Lagrange equation is
		\begin{equation}\label{eq:eigenfunction}
			\phi_j''+(V+\mu_j)\phi_j=0
			\qquad\text{in }I_j.
		\end{equation}
		Standard one-dimensional regularity implies that
		\begin{align*}
        \phi_j\in C^\infty(I_j).
        \end{align*} Since \(\phi_j\) is a nontrivial
		nonnegative first eigenfunction, the strong maximum principle yields
		\begin{align*}
		\phi_j(x)>0
		\qquad\text{for every }x\in I_j.
		\end{align*}
		
		Next,  we claim 
		\begin{equation}\label{eq:muj-zero}
			\lim_{j\to\infty}\mu_j=0.
		\end{equation}
		It is not hard to construct  a sequence of test functions \begin{align*}\eta_j\in C_c^\infty(I_j),\end{align*} such that
		\begin{align*}
		0\leq \eta_j\leq 1,
		\quad |\eta_j'|\leq 2, \quad 
		\eta_j\equiv 1
		\quad\text{on }[1,j-1].
		\end{align*}
	Since \(V\geq 0\), we obtain
		\begin{align*}
		0\leq \mu_j
		\leq
		\frac{
			\displaystyle
			\int_{I_j}
			\left(|\eta_j'|^2-V\eta_j^2\right)\ud x
		}{
			\displaystyle
			\int_{I_j}\eta_j^2\ud x
		}
		\leq
		\frac{
			\displaystyle
			\int_{I_j}|\eta_j'|^2\ud x
		}{
			\displaystyle
			\int_{I_j}\eta_j^2\ud x
		}\leq \frac{8}{j-2}
		\end{align*}
		which proves \eqref{eq:muj-zero} by letting $j\to\infty.$
		
		Fix the reference point \(x_0=1\). Since \(\phi_j(1)>0\), define
		\begin{align*}
		\psi_j(x):=\frac{\phi_j(x)}{\phi_j(1)}.
		\end{align*}
		Then, by using \eqref{eq:eigenfunction}, one has 
		\begin{align*}
		\psi_j(1)=1,
		\qquad
		\psi_j>0,
		\end{align*}
		and
		\begin{equation*}
			\psi_j''+(V+\mu_j)\psi_j=0
			\qquad\text{in }I_j.
		\end{equation*}
		
		With the help of standard elliptic theory and  Arzel\'a--Ascoli theorem, there exists a smooth and  nonnegative function $y(x)$ satisfies 
			\begin{align}\label{y''+Vy=0}
		y''+Vy=0
		\qquad\text{on }(0,\infty), \quad y(1)=1.
		\end{align}
		Finally, suppose that 
        \begin{equation}\label{yx*=0}
            y(x_*)=0
        \end{equation}
        for some \(x_*\in(0,\infty)\).
		Since \(y\geq 0\), the point \(x_*\) is a local minimum, and hence
		\begin{align}\label{y'x*=0}
		y'(x_*)=0.
		\end{align}
		The uniqueness theorem for the initial value problem  associated with
		\begin{align*}
		y''+Vy=0
		\end{align*}
        under \eqref{yx*=0} and \eqref{y'x*=0}
		then implies that \(y\equiv 0\), contradicting \eqref{y''+Vy=0}.  Therefore, one has 
		\begin{align*}
		y(x)>0
		\qquad\text{for every }x\in(0,\infty).
		\end{align*}
		This completes the proof.
	\end{proof}

	\section{ Strictly positivity of  $Q_g$ and $R_g$ implies compactness}\label{sec:positive-Q-and-R}

Under stronger conditions than \eqref{condition-of-Q-R} in Theorem \ref{thm:Bonnet myers}, we prove the following Bonnet-Myers type theorem.
	\begin{theorem}\label{thm:Bonnet myers-2}
		Let $(M^4,g)$ be a smooth, complete and connected   four-dimensional Riemannian manifold. If the $Q$-curvature $Q_g$ and the scalar curvature $R_g$  satisfy
        \begin{equation}\label{condtion-of-Q-2}
           Q_g\geq 6k^2 \quad  \mathrm{and}\quad   R_g\geq -12k+\varepsilon_0
        \end{equation}
        for some positive constants $k$ and $\varepsilon_0$, then $M^4$ is compact.
	\end{theorem}
\begin{proof} 
Based on the assumption \eqref{condtion-of-Q-2}, we  choose a constant
		$\varepsilon$ such that
        \begin{align*}
           \max\left\{0,12k-\frac{\varepsilon_0}{2}\right\}<\varepsilon<12k. 
        \end{align*}
		Then, one has 
		\begin{equation}\label{eq:R+epsilon}
		R_g+\varepsilon\ge \frac{\varepsilon_0}{2}>0
		\end{equation}
		and then 
		\begin{equation}\label{Q_g-lower-bound}
		    Q_g -\frac{\varepsilon^2}{24}
		\ge6k^2-\frac{\varepsilon^2}{24}=:c>0.
		\end{equation}
Under \eqref{eq:R+epsilon} and the completeness of $g$,  it is easy to see that  the conformal metric 
\begin{equation*}
    \tilde g= (R_g+\varepsilon)g
\end{equation*}
is also complete.

We claim that $M^4$ must be compact and argue by contradiction.

If $M^4$ is  noncompact, by the Hopf-Rinow theorem, $(M^4,\widetilde g)$ admits a
minimizing ray. Reparametrize this ray by $g$-arc length and denote it by
\begin{align}
\gamma:[0,\infty)\to M^4,
\qquad |\dot\gamma|_g=1.
\end{align}
The $g$-parameter interval is indeed infinite. Otherwise the ray would
have finite $g$-length, and the completeness of $g$ would force it to
converge to a point of $M^4$, contradicting the fact that a
$\widetilde g$-ray escapes every compact set.

Apply Lemma \ref{lem:SY} with
\begin{align*}
n=4,
\quad f=R_g+\varepsilon,
\quad \sigma=\frac12,
\quad \tilde g=(R_g+\varepsilon)g.
\end{align*}
Set	\begin{align*}z(s):=\log (R_g(\gamma(s))+\varepsilon).\end{align*}
Then,  for every
$\phi\in C_c^\infty((0,\infty))$, one has 
\begin{align*}\begin{split}
	&3\int(\phi')^2
	+\frac{3}{16}\int(z')^2\phi^2
	-\frac12\int z'\phi\phi'
	\\&\qquad\ge{}
	\int \Ric^{R_g+\varepsilon,\frac12}_g
(\dot\gamma,\dot\gamma)\phi^2+\frac12\int|\nabla z|_g^2\phi^2
\end{split}
\end{align*}
where  we omit $\ud s$ for brevity if no confusion causes.

Using the estimate \begin{align}|\nabla z|_g^2\ge(z')^2, \end{align}\eqref{eq:conf-Ric-R} and \eqref{Q_g-lower-bound}, we find  that 
\begin{equation}\label{eq:non-negative operator}
	3\int(\phi')^2
	-\frac12\int z'\phi\phi'
	-\frac5{16}\int(z')^2\phi^2
	\geq 3c\int e^{-z}\phi^2.
\end{equation}

Applying Young's inequality, one has
\begin{align}\label{Young's ineqaulity}
-\frac12z'\phi\phi'
\le (\phi')^2+\frac1{16}(z')^2\phi^2.
\end{align} 
Combining with \eqref{eq:non-negative operator}, we obtain that,
for every $\phi\in C_c^\infty((0,\infty))$, 
\begin{equation}
	\int_0^\infty(\phi')^2\ud s
	\ge
	\int_0^\infty
	\left(\frac1{16}(z')^2+\frac{3c}{4}e^{-z}\right)
	\phi^2\ud s
\end{equation}
Set a positive function
\begin{align*}
    V:=\frac1{16}(z')^2+\frac{3c}{4}e^{-z}>0.
\end{align*}
Using  Lemma \ref{lem:positive-solution}, there exists a smooth function $y>0$ satisfying
	\begin{equation}\label{eq:yODE}
		y''+Vy=0.
	\end{equation}
	Since $V>0$, we have $y''=-Vy<0$, so $y$ is strictly concave. A concave
	function that remains positive on the whole half-line must satisfy
\begin{equation}\label{y'geq0}
    y'\ge0.
\end{equation}
 Indeed, if $y'<0$ at some point, then the monotonicity of $y'$
	would force $y$ to cross zero in finite time.
	
Then,  we set	\begin{equation*}
p:=\frac{y'}{y}\ge0, \quad	w:=e^{-z/2}>0.\end{equation*}
	Combining with \eqref{eq:yODE}, we obtain that 
	\begin{align*}
p'+p^2+\frac{(w')^2}{4w^2}+\frac{3c}{4}w^2=0.
	\end{align*}
	Now,   we set
    \begin{equation*}
        a:=\frac{p}{w}\ge0.
    \end{equation*}
A direct computation gives
\begin{equation}	\label{eq:aprime}
	a'=\frac{p'}w-\frac{pw'}{w^2}=-\frac1{4w}
	\left(\frac{w'}w+2aw\right)^2-\frac{3c}{4}w.
\end{equation}
	Thus $a$ is nonincreasing and nonnegative.  Integrating
	\eqref{eq:aprime} from $1$ to infinity, we get
	\begin{equation}\label{eq:wL1}
		\int_{1}^\infty w\ud s
		\le\frac{4}{3c}a(1)<\infty,
	\end{equation}
	and
	\begin{equation}\label{eq:bL2}
		\int_{1}^\infty
		\frac1w
		\left(\frac{w'}w+2aw\right)^2\ud s
		\le4a(1)<\infty.
	\end{equation}
Using  H\"older's inequality,
	\eqref{eq:wL1}, and \eqref{eq:bL2}, there holds
	\begin{align*}
		\begin{split}&\quad\int_{1}^\infty\left|\frac{w'}w+2aw \right|\ud s\\
		&\leq
		\left(\int_{1}^\infty \frac{1}{w}\left(\frac{w'}w+2aw\right)^2 \ud s\right)^{1/2}
		\left(\int_{1}^\infty w \ud s\right)^{1/2}<\infty.
    \end{split}
	\end{align*}
	Since $a$ is nonnegative and nonincreasing, it is bounded. Together with
	$w\in L^1((1,\infty))$, this gives
$aw\in L^1((1,\infty)).$
	Then, one has
	$\frac{w'}{w}\in L^1((1,\infty)).$
 Hence $\log w(s)$ converges to a
	finite real number as $s\to\infty$, and then
	\begin{align}
	w(s)\longrightarrow w_\infty>0, \quad \mathrm{as}\; s\to\infty
	\end{align}
    for some positive constant $w_\infty$ which  contradicts \eqref{eq:wL1}.	Therefore, $M^4$ must be  compact.

Thus, we finish the proof of  Theorem
	\ref{thm:Bonnet myers-2}. 
\end{proof}

    As we have shown during the proof of Theorem \ref{thm:Bonnet myers-2}, although we only assume that
    \begin{align*}
       R_g\geq -12k+\varepsilon_0,
    \end{align*} the structure of $Q$-curvature together with its lower bound actually  forces $R_g$ to have  a sharp positive lower bound. This observation plays a crucial role in the proof of gap phenomenon for $R_g$ in the next section. For the reader's convenience, we state this as a corollary.
\begin{corollary}\label{cor:R_g-lower-bound}
    Under the same assumptions of Theorem \ref{thm:Bonnet myers-2}, one has
    \begin{align*}R_g\geq 12k.\end{align*}
\end{corollary}	
\begin{proof}
Using Theorem \ref{thm:Bonnet myers-2}, $M^4$ is compact.
    Let $x_0$ be a minimum point of $R_g$. Rewrite
		\eqref{def of Q} as
        \begin{align*}
        6 Q_g= -\Delta_g R_g-3\left|\Ric_g-\frac{R_g}{4}g\right|_g^2+\frac{R_g^2}{4}.\end{align*}
		At the minimum point $x_0$, one has $\Delta_gR_g(x_0)\ge0$ and then one has
		\begin{align*}
        36k^2\leq 6 Q_g(x_0)\le\frac{1}{4}R_g(x_0)^2.
		\end{align*}
		Hence, we obtain 
        $$|R_g(x_0)|\geq 12k.$$ 
        Based on our assumption \eqref{condtion-of-Q-2}, one has 
        \begin{equation*}
           R_g(x_0)\geq -12k+\varepsilon_0>-12k. 
        \end{equation*}
    Thus, one has 
		\begin{equation*}
		R_g(x_0)\geq 12k. 
		\end{equation*} 
        Thus, we finish the proof.
\end{proof}

\begin{prop}\label{pro:p_1-formula}
 Let $(M^4,g)$ be a  compact four-dimensional Riemannian manifold  satisfying
    \eqref{eq:positive-Q-and-Y}. 
    Then, the fundamental group $\pi_1(M^4)$ is finite and satisfies \eqref{pi_1-equation}.
\end{prop}
\begin{proof}
    Under our assumptions, Corollary B  of \cite{CGY} shows that there exists a conformal metric $ g_1\in [g]$ such that 
\begin{align*}
\Ric_{g_1}>0, \quad \forall x\in M^4
\end{align*} 
Since $M^4$ is compact, there exists a positive constant $d_1$ such that 
\begin{equation}\label{Ric_g_1geqd1g}
    \Ric_{g_1}\geq d_1 g_1.
\end{equation}
Consider the universal cover of $M^4$ denoted as $\widetilde M^4$:
\begin{align*}
\pi: (\widetilde M^4, \tilde g_1)\longrightarrow (M^4,g_1).
\end{align*}
Then, the lower bound \eqref{Ric_g_1geqd1g} for the Ricci tensor of  $(\widetilde M^4, \tilde g_1)$ stills holds.
Applying Bonnet-Myers theorem, $\widetilde M^4$ is compact. Then, using the conformal invariant property of $Q$-curvature integral \eqref{eq:Q-conformal-invar} and the  Gursky's inequality \eqref{Gursky upper bound} (See Theorem B in \cite{Gursky CMP}), one has 
\begin{align*}\begin{split}
&\quad|\pi_1(M^4)|\int_{M^4}Q_g\ud\mu_g\\
&=|\pi_1(M^4)|\int_{M^4}Q_{g_1}\ud\mu_{g_1}\\
&=\int_{\widetilde M^4}Q_{\tilde g_1}\ud\mu_{\tilde g_1}\leq 16\pi^2,\end{split}\end{align*}
which finishes our proof.
\end{proof}

{\bf Proof of Corollary \ref{cor:pi_1-control-for-Q}:}
Using Theorem \ref{thm:Bonnet myers-2},  $M^4$ is compact. Combining with \eqref{cor:R_g-lower-bound}, one has $R_g>0.$ Then, $(M^4,g)$ satisfies \eqref{eq:positive-Q-and-Y}. Then, our conclusion follows from Proposition \ref{pro:p_1-formula}. \qed

\section{Completeness of conformal metric and $R_g$ gap phenomenon}\label{sec:compelte-metric}

In this section, we now turn to the critical  assumption
\begin{equation}\label{Q-R-lower-bound}
    Q_g\ge6k^2,\qquad R_g\ge-12k
\end{equation}
where $k$ is  a constant satisfying 
$$k\geq 0.$$
For brevity, set the following notations
\begin{equation}\label{eq:SFE}
S:=R_g+12k\ge0,
\qquad
F:=Q_g-6k^2\ge0,
\qquad
E:=\Ric_g-\frac{R_g}{4}g.
\end{equation}
Using the definition \eqref{def of Q}, a direct computation shows that 
\begin{equation}\label{eq:S-equation}
\Delta_g S
=
-6kS+\frac14S^2-3|E|_g^2-6F.
\end{equation}

The following lemma is motivated by Lemma 2.1 of \cite{GM}, originally due to Gursky and Malchiodi. Although elementary, it is quite subtle and plays a vital role in the proof of the strong maximum principle for the Paneitz operator. For further details, we refer the reader to \cite{GM}.
\begin{lemma}\label{lem:GM-lemma}
Suppose that  \eqref{Q-R-lower-bound} holds  and the constant $k\geq0$.  If \(R_g(p)=-12k\) at some point \(p\in M^4\), then one has
\begin{align*}
 \Ric_g\equiv-3kg.
\end{align*}
\end{lemma}

\begin{proof}
First, if $k>0$, since \(S\ge0\) and \(S(p)=0\), there is a connected neighborhood
\(U\) of \(p\) on which \(S<24k\). On \(U\), using \eqref{eq:S-equation}, one has 
\begin{align*}
\Delta_g S
=
-\left(6k-\frac S4\right)S-3|E|_g^2-6F
\le0.
\end{align*}
Thus,  \(S\) is superharmonic and attains its minimum zero at an interior
point. The strong maximum principle yields that  \(S\equiv0\) on \(U\).
Hence the zero set of \(S\) is both open and closed. Since \(M^4\) is
connected, \(S\equiv0\) on \(M^4\). Then, the equation
\eqref{eq:S-equation}  gives \(E\equiv0\) and \(F\equiv0\). Returning back to \eqref{eq:SFE}, one has
$$\Ric_g\equiv -3kg.$$

Second, if $k=0$, one has
\begin{align*}
    R_g\geq 0,  \quad  -\Delta_g R_g+\frac{R_g^2}{4}\geq 0.
\end{align*}
Since $R_g(p)=0,$  the strong maximum principle yields that $R_g\equiv 0$. Then, the equation
\eqref{eq:S-equation} yields $E\equiv 0$ and $Q_g\equiv 0.$ Then, $\Ric_g\equiv 0.$

Thus, we finish our proof.
\end{proof}

From now on, we assume that 
\begin{equation}\label{eq:Fgeq0-S-positive}
F\geq 0, \quad S>0.
\end{equation} 
Consider the conformal metric 
\begin{equation}\label{eq:def-of-gS}
g_{S}:=Sg.
\end{equation}
Taking $\varepsilon=12k$ in Lemma \ref{lem:modified CGYlemma}, we obtain
\begin{equation}\label{eq:B-nonnegative}
\Ric_g^{S,1/2}=
\Ric_g-\frac{\Delta_gS}{2S}g
\ge
3\frac{Q_g-6k^2}{S}g
=
3\frac FSg
\ge0.
\end{equation}

The first issue is that \(S\) may tend to zero at infinity, so
completeness of \(g_S\) is not automatic.  We will establish a lemma about completeness of \(g_S\) for subsequent use.

We first introduce some notations. Let  $d_{g}(p,x)$ denote the the $g$-geodesic distance between $x$ and $p$ and   $B_g(p,\rho)$ denote the geodesic ball with respect $g$ centering at $p$ with radius $\rho$.

To avoid notational clutter, we shall employ the same symbols  as before throughout the proof, provided that no ambiguity arises.

\begin{lemma}\label{lem:g_S-complete}
Suppose that \eqref{eq:Fgeq0-S-positive} holds. Then, the conformal  metric
$g_S$ from \eqref{eq:def-of-gS} is  complete.
\end{lemma}
\begin{proof}
If $M^4$ is compact, the completeness of $g_S$ holds automatically. 

From now on, we deal with the situation $M^4$ is noncompact.
We argue by contradiction and suppose that $g_S$ is incomplete. Then, its metric completion contains a boundary
point. Equivalently, there is an \(g_S\)-Cauchy sequence with no limit in
\(M^4\). Such a sequence must leave every compact subset of \(M^4\):
on each \(g\)-compact set the positive function \(S\) has positive
upper and lower bounds, so \(g\) and \(g_S\) are uniformly equivalent.
By passing to a subsequence and connecting successive points by curves
of summable \(g_S\)-length, we obtain a curve
\begin{align*}
\gamma_0:[0,1)\to M^4
\end{align*}
which leaves every \(g\)-compact set and has finite total \(g_S\)-length $\mathrm{Length}_{g_S}(\gamma_0).$
Fix a point $$p=\gamma_0(0).$$
For any  \(\rho>0\),  the curve \(\gamma_0\) meets \(\partial B_g(p,\rho)\), and hence
\begin{align*}
d_{g_S}\bigl(p,\partial B_g(p,\rho)\bigr) \le \mathrm{Length}_{g_S}(\gamma_0).
\end{align*}

The closed \(g\)-ball \(\overline B_g(p,\rho)\) is compact. Since \(g_S\)
is a smooth positive metric on it, there is an \(g_S\)-length minimizing
curve from \(p\) to \(\partial B_g(p,\rho)\). Stop it at the first
intersection with the boundary and denote it by \(\gamma_\rho\).
Parametrize \(\gamma_\rho\) by \(g\)-arc length. Its \(g\)-length is at
least \(\rho\), whereas its \(g_S\)-length is bounded above by
\(\mathrm{Length}_{g_S}(\gamma_0)\). Choose a sequence \(\rho_j\to\infty\). The initial \(g\)-unit tangent vectors lie
in the compact unit sphere of \(T_pM^4\). After passing to a subsequence,
they converge. The reparametrized \(g_S\)-geodesic equation is a smooth
ODE on every fixed \(g\)-compact set; hence the curves converge smoothly
on every bounded \(g\)-arc-length interval to a curve
\begin{align}
\gamma:[0,\infty)\to M,
\qquad |\dot\gamma|_g=1.
\end{align}
 For every \(T>0\),
the \(g_S\)-length of \(\gamma|_{[0,T]}\) is at most \(\mathrm{Length}_{g_S}(\gamma_0)\).
Letting \(T\to\infty\), we obtain that  
\begin{equation}\label{eq:finite-gS-ray}
\int_0^\infty\sqrt{S(\gamma(s))}\ud s<\infty.
\end{equation}
Set \begin{align*} 
z(s):=\log S(\gamma(s)).
\end{align*}
Then, \eqref{eq:finite-gS-ray} is equivalent to 
\begin{equation}\label{ezinL^1}
    e^{z/2}\in L^1((0,+\infty)).
\end{equation}

Applying  Lemma \ref{lem:SY} with
\begin{align*}
n=4, \quad f=S, \quad \sigma=\frac12,
\end{align*}
and using \eqref{eq:B-nonnegative},
we obtain, for every
\(\varphi\in C_c^\infty((0,\infty))\),
\begin{equation*}
3\int(\varphi')^2
-\frac12\int z'\varphi\varphi'
-\frac5{16}\int(z')^2\varphi^2
\ge0
\end{equation*}
Using Young's inequality \eqref{Young's ineqaulity},  one has
\begin{equation}
\int(\varphi')^2
\ge
\frac1{16}\int(z')^2\varphi^2.
\end{equation}
Combining with Lemma \ref{lem:positive-solution},  there exists  a positive solution $y$  to the following equation
\begin{equation}\label{eq:y-equation-for-gS}
y>0,\quad y''+\frac{1}{16}\left(z'\right)^2y=0, \quad \mathrm{on}\; (0,+\infty).
\end{equation}
Due to the same reason for \eqref{y'geq0}, one has \(y'\ge0\).
Set
\begin{equation}\label{Pgeq0-Ageq0}
P:=\frac{y'}y\ge0,
\qquad
A:=\frac {P}{e^{z/2}}\ge0.
\end{equation}
Then, a direct computation yields that \eqref{eq:y-equation-for-gS} is equivalent to
\begin{align}
P'+P^2+\frac{1}{16}\left(z'\right)^2=0,
\end{align}
and then
\begin{equation}\label{eq:A-critical}
A'
=
-e^{-z/2}
\left(
P+\frac{1}{4}z'
\right)^2
\le0.
\end{equation}
Combining with \eqref{Pgeq0-Ageq0}, one has $A$ must be nonincreasing and bounded.
 Integrating
\eqref{eq:A-critical} from $1$ to infinity,  one has 
\begin{align}
\int_{1}^{\infty}
e^{-z/2}
\left(
P+\frac{1}{4}z'
\right)^2\ud s\leq A(1)<\infty.
\end{align}
Combining with \eqref{ezinL^1} and H\"older's inequality, there holds 
\begin{equation}\label{eq:P+logv-in-L1}
    P+\frac{1}{4}z'\in L^1([1,\infty)).
\end{equation}
Since  \(A\)  is bounded,   \eqref{ezinL^1}  yields that 
\begin{align*}
\int_{1}^{\infty}P\ud s
=\int_{1}^{\infty}Ae^{z/2}\ud s
<\infty.
\end{align*}
Therefore, combining with \eqref{eq:P+logv-in-L1}, one has  
\begin{align*}
z'\in L^1([1,+\infty)),
\end{align*} 
and then  \(z(s)\) converges to a finite real number 
 which contradicts \eqref{ezinL^1}.
 Therefore,  \(g_S\) must be  complete and we finish our proof.
\end{proof}

In the remainder of this section, we establish a gap phenomenon for the scalar curvature \(R_g\); see Proposition \ref{prop:gap-R}.

For brevity, set $$f=\log S$$ and then  the conformal metric  \eqref{eq:def-of-gS} can be rewritten as 
\begin{align*}
g_S=e^fg.
\end{align*}
Moreover, $g_S$ is complete due to Lemma \ref{lem:g_S-complete}. Consider  the weighted Laplacian operator 
\begin{equation}\label{eq:L-def}
\mathcal{L}\phi
=
\Delta_{g_S}\phi-\langle\nabla^{g_S}f,\nabla^{g_S}\phi\rangle_{g_S}.
\end{equation}
Based on this definition, for every smooth function \(\phi\) on $M^4$, a direct computation yields that 
\begin{equation}\label{eq:L-formula}
\mathcal{L}\phi=S^{-1}\Delta_g\phi.
\end{equation}
For 4-manifolds,  the 
\(m\)-Bakry-\'Emery Ricci tensor is given by 
\begin{align}\label{m-Back-Em}
\Ric_f^m(g_S)
:=
\Ric_{g_S}+\Hess_{g_S}f-\frac{1}{m-4}df\otimes df.
\end{align}
More discussions can be found in \cite{KL}. We choose $m=0$ in \eqref{m-Back-Em} and obtain 
\begin{equation}\label{eq:Ricf0-def}
\Ric_f^0(g_S)
=
\Ric_{g_S}+\Hess_{g_S}f+\frac14df\otimes df
\end{equation}
and then a useful lemma follows.
\begin{lemma}\label{lem:Ricf0}
Suppose that \eqref{eq:Fgeq0-S-positive} holds. Then one has
\begin{align}
\Ric_f^0(g_S)\ge0.
\end{align}
\end{lemma}
\begin{proof}
The conformal transformation formula for $g_S$ gives
\begin{align}
\Ric_{g_S}
=
\Ric_g-\Hess_gf+\frac12df\otimes df
-\frac12\left(\Delta_gf+|\nabla^gf|_g^2\right)g.
\end{align}
Meanwhile, one has 
\begin{align}
\Hess_{g_S}f
=
\Hess_gf-df\otimes df
+\frac12|\nabla^gf|_g^2g.
\end{align}
Combining with these two equations, we obtain
\begin{align}\begin{split}
\Ric_f^0(g_S)
&=
\Ric_g-\frac12(\Delta_gf)g-\frac14df\otimes df\\
&=
\left(
\Ric_g-\frac{\Delta_gS}{2S}g
\right)
+\left(\frac12|\nabla^gf|_g^2g
-\frac14df\otimes df\right).\end{split}
\end{align}
The first term in the above equation is nonnegative by \eqref{eq:B-nonnegative} while the second term is nonnegative obviously.
Hence, we finish our proof.
\end{proof}

Building on Lemma \ref{lem:g_S-complete} and Lemma \ref{lem:Ricf0}, we obtain
the  Laplacian  comparison theorem by using Theorem 2.2 of \cite{KL}. The following lemma is just a special case of Kuwae-Li's result \cite{KL}.
For the reader's convenience, we sketch their proof here.

Before stating  the lemma, we introduce some notations. For some fixed point $p$,   and  $\operatorname{Cut}_{g_S}(p)$ denotes the cut locus for the point $p$  with respect to the metric $g_S$. 
For brevity, set 
\begin{align}\label{equ:def-of-r_p}
r_p(x):=d_{g_S}(p,x),
\end{align}

\begin{lemma}\label{lem:Laplacian-comparison}
For fixed \(p\in M^4\) and take $x\notin\{p\}\cup\operatorname{Cut}_{g_S}(p)$. 
Let
\begin{align*}
\gamma:[0,r_p(x)]\to M^4
\end{align*} be the unique \(g_S\)-unit-speed minimizing
geodesic from \(p\) to \(x\), and define
\begin{equation}\label{eq:s-p-def}
s_p^{\gamma}(x):=\int_0^{r_p(x)}e^{-f(\gamma(t))/2}\ud t.
\end{equation}
Then, one has 
\begin{equation}\label{eq:Lr}
\mathcal{L}(r_p(x))
\le
\frac{4e^{-f(x)/2}}{s_p(x)}.
\end{equation}
\end{lemma}
\begin{proof}
For brevity, set the notations as follows 
\begin{align*}
m_f(t):=L(r_p(\gamma(t))),
\qquad
q(t):=\frac{d}{dt}(f\circ\gamma)(t).
\end{align*}
Then, using \eqref{eq:L-def}, one has
\begin{align*}
m_f=\Delta_{g_S}r_p-q.
\end{align*}
The radial Riccati identity gives
\begin{align*}
\frac{d}{dt}\Delta_{g_S}r_p
=
-|\Hess_{g_S}r_p|_{g_S}^2
-\Ric_{g_S}(\dot\gamma,\dot\gamma).
\end{align*}
Combining these two identities, we obtain
\begin{align}
m_f'
&=
-|\Hess_{g_S}r_p|_{g_S}^2
-\Ric_{g_S}(\dot\gamma,\dot\gamma)
-f''.
\label{eq:mf-exact}
\end{align}
The definition \eqref{eq:Ricf0-def} and   Lemma \ref{lem:Ricf0} imply that 
\begin{align}
-\Ric_{g_S}(\dot\gamma,\dot\gamma)-f''
\le
\frac14q^2.
\end{align}
Moreover, the radial eigenvalue of \(\Hess_{g_S}r_p\) is zero, so one has 
\begin{align}
|\Hess_{g_S}r_p|_{g_S}^2
\ge
\frac{(\Delta_{g_S}r_p)^2}{3}
=
\frac{(m_f+q)^2}{3}.
\end{align}
Inserting   these two estimates  into \eqref{eq:mf-exact} yields
\begin{equation}\label{eq:mf-ineq}
m_f'
\le
-\frac{(m_f+q)^2}{3}
+\frac14q^2.
\end{equation}

For brevity, set another two  notations
\begin{align}
\Lambda=e^{f/2}m_f,
\qquad
s(t)=\int_0^te^{-f(\gamma(\tau))/2}\,d\tau.
\end{align}
Immediately, one has 
 \begin{align}
 \frac{\ud s}{\ud t}=e^{-f/2},
 \end{align} 
 and then
\begin{align}\label{Lmabda'}
\frac{d\Lambda}{ds}
=
e^f\left(m_f'+\frac12qm_f\right).
\end{align}
Using \eqref{eq:mf-ineq}, there holds
\begin{align}\begin{split}
m_f'+\frac12qm_f
&\le
-\frac13m_f^2-\frac16qm_f-\frac1{12}q^2\\
&=
-\frac14m_f^2-\frac1{12}(m_f+q)^2\\
&\le
-\frac14m_f^2.\end{split}
\end{align}
Consequently, combining with \eqref{Lmabda'},  we obtain
\begin{equation}\label{eq:Lambda-ineq}
\frac{d\Lambda}{ds}
\le
-\frac14\Lambda^2.
\end{equation}

The standard distance asymptotics at the  point $p$  give
\begin{align}
m_f(t)=\frac3t+O(1),
\qquad
s(t)=e^{-f(p)/2}t+O(t^2),
\end{align}
and hence
\begin{equation}\label{eq:Lambda-asymptotic}
\Lambda(s)=\frac3s+O(1)
\qquad\text{as }s\rightarrow 0^+.
\end{equation}
Fix an endpoint value of \(s\). If \(\Lambda\) becomes nonpositive
before that endpoint, the desired upper bound is immediate. Otherwise
\(\Lambda>0\) on the relevant interval, and
\eqref{eq:Lambda-ineq} yields that 
\begin{align}\label{Lambda^-1'}
    (\Lambda^{-1})'\geq 1/4.
\end{align}
Integrating  \eqref{Lambda^-1'} from \(s_0>0\) to $s$ and then using
\eqref{eq:Lambda-asymptotic},  as \(s_0\rightarrow 0^+\), we obtain
\begin{align}
   \Lambda(s)^{-1}\ge s/4 
\end{align}
which  proves
\eqref{eq:Lr}. Thus, we finish the proof.
\end{proof}

The classical Calabi's preceding estimate  \cite{Calabi} extends to the cut locus in the upper-barrier
sense. For the reader's convenience, we record the precise form needed below.

\begin{lemma}\label{lem:Calabi-support}
Let \(x\ne p\) and 
\begin{align*}
\gamma:[0,\ell]\to M^4, \quad \ell=r_p(x),
\end{align*} be any
\(g_S\)-unit-speed minimizing geodesic from \(p\) to \(x\), and consider  $s_p^{\gamma}(x)$ same as \eqref{eq:s-p-def}.
For every \(\varepsilon\in(0,\ell)\), set 
\begin{align}\label{r_p_epsilon}
p_\varepsilon=\gamma(\varepsilon),
\qquad
r_{p,\varepsilon}(y)
=
\varepsilon+d_{g_S}(p_\varepsilon,y).
\end{align}
Then, one has 
\begin{align*}
x\notin\operatorname{Cut}_{g_S}(p_\varepsilon).
\end{align*} In particular,
\(r_{p,\varepsilon}\) is smooth in a neighborhood of \(x\), and  one has 
\begin{equation*}
\mathcal{L}r_{p,\varepsilon}(x)
\le
\frac{4e^{-f(x)/2}}
{s_p^{\gamma}(x)-s_p^{\gamma}(p_\varepsilon)}.
\end{equation*}
\end{lemma}

\begin{proof}
For every \(y\in M^4\), one has 
\begin{align*}
r_p(y)
=d_{g_S}(p,y)
\le
 d_{g_S}(p,p_\varepsilon)+d_{g_S}(p_\varepsilon,y)
=
r_{p,\varepsilon}(y).
\end{align*}
Since the terminal segment \(\gamma|_{[\varepsilon,\ell]}\) has length
\(\ell-\varepsilon\), equality holds at \(x\).

We claim  that 
\begin{align}\label{x-not-in-cut-locus}
x\notin\operatorname{Cut}_{g_S}(p_\varepsilon).
\end{align}

First,
the terminal segment is the unique minimizing geodesic from
\(p_\varepsilon\) to \(x\). Indeed, if \(\sigma\) were another
minimizing geodesic, then the concatenation of
\(\gamma|_{[0,\varepsilon]}\) with \(\sigma\) would have total length
$\ell=d_{g_S}(p,x)$ and hence would itself be a minimizing curve from
\(p\) to \(x\). Every length-minimizing curve in a smooth Riemannian
manifold is a smooth geodesic, so the two pieces must have the same
tangent at \(p_\varepsilon\). Uniqueness for the geodesic equation then
forces $\sigma=\gamma|_{[\varepsilon,\ell]}.$ It remains to rule out conjugacy. Suppose that \(x\) were conjugate to
\(p_\varepsilon\) along the terminal segment  which shows that there exists a conjugate point $p_\varepsilon$ between  the points $p$ and $x$ along the local minimizing curve $\gamma_{[0,\ell]}$ which contradicts to Jacobi's theorem (see, for example \cite[Theorem 6.3.1]{Jost}).

 Thus, we prove the claim \eqref{x-not-in-cut-locus}. 
Then,  \(d_{g_S}(p_\varepsilon,\cdot)\) is smooth in a neighborhood of \(x\).
By applying  Lemma~\ref{lem:Laplacian-comparison} with base point
\(p_\varepsilon\) and  using the fact 
\begin{align*}
\int_\varepsilon^\ell e^{-f(\gamma(t))/2}\,dt
=
s_p^{\gamma}(x)-s_p^{\gamma}(p_\varepsilon), 
\end{align*}
we finish the proof.
\end{proof}

 If
\(\gamma\) is any \(g_S\)-unit-speed curve, then
\begin{align}
   |\dot\gamma|_g=e^{-f/2}. 
\end{align}
Thus the integral of \(e^{-f/2}\) with respect to \(g_S\)-arc length is
exactly the \(g\)-length. In particular, for every minimizing
\(g_S\)-geodesic \(\gamma\) from \(p\) to \(x\), one has 
\begin{equation}\label{eq:s-ge-dg}
s_p^{\gamma}(x) \ge
d_g(p,x).
\end{equation}
Then,  for every minimizing \(g_S\)-geodesic \(\gamma\) from \(p\) to \(x\), if \begin{align}r_p(x)\geq 1,\end{align} we have 
\begin{equation}\label{eq:delta-p}
s_p^{\gamma}(x)\geq s_p^{\gamma}(\gamma(1))\ge d_g(p,\partial B_{g_S}(p,1))=:\delta_p>0
\end{equation}

\begin{lemma}\label{lem:q-inequality}
Let \begin{align}\label{eq:q-and-v-def}
q:=S^{-1/2},\quad v:=\frac{q}{1+q}.
\end{align}
Then, one has 
\begin{equation}\label{eq:v-inequality}
\mathcal{L}v\ge\frac{3kq^3-\frac18q}{(1+q)^2}.
\end{equation}
\end{lemma}

\begin{proof}
Using \eqref{eq:L-formula}, one has 
\begin{align}
\mathcal{L}S^{-1}=S^{-1}\Delta_g(S^{-1})
=
-S^{-3}\Delta_gS+2S^{-4}|\nabla^g S|_g^2.
\end{align}
Combining with  \eqref{eq:SFE} and \eqref{eq:S-equation}, there holds
\begin{align}\begin{split}
\mathcal{L}S^{-1}
&=
6kS^{-2}-\frac14S^{-1}
+3|E|_g^2S^{-3}+6FS^{-3}
+2S^{-4}|\nabla^gS|_g^2\\
&\ge
6kS^{-2}-\frac14 S^{-1}+ 2S|\nabla^{g_S}S^{-1}|_{g_S}^2.\end{split}
\end{align}
Inserting $S^{-1}=q^2$, we obtain
\begin{align}
2q\mathcal{L}(q)+2|\nabla^{g_S}q|_{g_S}^2
\ge
6kq^4-\frac14q^2
+8|\nabla^{g_S}q|_{g_S}^2.
\end{align}
which yields that 
\begin{align}
\mathcal{L}q
\ge
3kq^3-\frac18q
+\frac3q|\nabla^{g_S}q|_{g_S}^2.
\end{align}
Then, a direct computation yields that 
\begin{align}
\begin{split}
\mathcal{L}\left(v\right)&=\frac{\mathcal{L}q}{(1+q)^2}
-\frac{2|\nabla^{g_S}q|_{g_S}^2}{(1+q)^3}\\
&\ge\frac{3kq^3-\frac18q}{(1+q)^2}
+
\frac{q+3}{q(1+q)^3}
|\nabla^{g_S}q|_{g_S}^2.\end{split}
\end{align}
Thus, we finish our proof.
\end{proof}

Building the above preparations, we are now ready to prove  the following proposition by employing the famous  Omori-Yau type argument.

\begin{prop}\label{prop:gap-R}
  Suppose that \eqref{eq:Fgeq0-S-positive} holds and the constant $k>0$. Then, one has 
  \begin{equation}\label{Rggeq12k}
      R_g\geq 12k.
  \end{equation}
\end{prop}
\begin{proof}
If $M^4$ is compact, Corollary \ref{cor:R_g-lower-bound} shows that \eqref{Rggeq12k} holds.  Thus, we only need to deal with  the noncompact case.
Fix  a point \(p\in M\), recall the notations \eqref{equ:def-of-r_p}
and \eqref{eq:q-and-v-def}
\begin{align*}
r_p(x)=d_{g_S}(p,x), \quad q=S^{-1/2},
\qquad
v=\frac{q}{1+q}.
\end{align*}
Obviously, one has 
\begin{align}\label{v-bounded}
0<\sup_{M^4}v \leq 1.
\end{align}

If  \(\sup_{M^4} v\) is attained at some point
\(x_0\).  Thus, one has 
\begin{align*}
\mathcal{L}v(x_0)= S^{-1}\Delta_gv(x_0)\leq 0.
\end{align*}
Combining with \eqref{eq:v-inequality}, we obtain
\begin{align*}
0
\ge
\frac{q(x_0)
\left(3kq(x_0)^2-\frac18\right)}
{(1+q(x_0))^2}
\end{align*}
which yields that 
\begin{align*}
q(x_0)^2\le\frac1{24k}.
\end{align*}
Since \(v\) is a strictly increasing function of \(q\),
the function \(q\) also attains its maximum at \(x_0\) and then $R_g$ attains its minimum at $x_0$.
Thus,  \eqref{Rggeq12k} holds.

Suppose now that \(\sup_Mv\) is not attained. Choose a sequence \(y_j\in M^4\) such that
\begin{align}
v(y_j)>\sup_{M^4} v-\frac1j,
\end{align}
and set
\begin{align}
\delta_j
:=
\frac1{j(1+r_p(y_j))}.
\end{align}
Since \(v(x)\) is bounded \eqref{v-bounded} and \(r_p(x)\) is proper on the complete manifold
\(({M^4},g_S)\)  (Lemma \ref{lem:g_S-complete}), the function
\begin{align}\label{v-deltar_p}
v(x)-\delta_jr_p(x)
\end{align}
attains its maximum at some point \(x_j\). By the definition of
\(x_j\), there holds
\begin{align*}
\begin{split}
v(x_j)-\delta_jr_p(x_j)
&\ge
v(y_j)-\delta_jr_p(y_j)\\
&>
\sup_{M^4} v-\frac2j.\end{split}
\end{align*}
Hence, we have 
\begin{equation*}
v(x_j)>\sup_{M^4} v-\frac2j.
\end{equation*}
In particular, one has 
\begin{align}\label{v-to-sup_Mv}
v(x_j)\longrightarrow \sup_{M^4} v.
\end{align}
Since \(\sup_{M^4} v\) is not attained, the sequence \(x_j\) leaves every compact
subset of \(M^4\). Thus, for all sufficiently large \(j\), one has 
\begin{align*}
r_p(x_j)\ge1.
\end{align*}
Let \(\gamma_j\) be a minimizing \(g_S\)-geodesic from \(p\) to \(x_j\).
Combining with the  estimate \eqref{eq:delta-p}, one has
\begin{align*}
    s_p^{\gamma_j}(x_j)\ge\delta_p.
\end{align*}
If $x_j$ lies in the cut locus of $p$, we use $r_{p,\varepsilon}$ \eqref{r_p_epsilon} to replace $r_p$ and choose $\varepsilon$ in Lemma \ref{lem:Calabi-support}  small enough such that 
\begin{align*}
    s_p^{\gamma_j}(p_\varepsilon)\leq \frac{\delta_p}{2}.
\end{align*}
Thus, using Lemma \ref{lem:Laplacian-comparison} and Lemma \ref{lem:Calabi-support}, one has
\begin{align*}
\mathcal{L}(r_p(x_j))\leq \frac{8}{\delta_p}q(x_j).
\end{align*}
Since  \eqref{v-deltar_p} attains maximum  at \(x_j\),  
 one has
\begin{align*}
\mathcal{L}(v-\delta_jr_p)(x_j)=S^{-1}\Delta_g(v-\delta_j r_p)(x_j)\leq 0
\end{align*}
which yields that 
\begin{equation}\label{eq:penalized-upper}
\mathcal{L}v(x_j)\leq \delta_j\mathcal{L}(r_p(x_j))
\le \frac{8}{\delta_p}\delta_j q(x_j).
\end{equation}
Combining \eqref{eq:penalized-upper} with \eqref{eq:v-inequality}, one has  gives
\begin{align*}
\frac{3kq(x_j)^3-\frac18q(x_j)}{(1+q(x_j))^2}
\le
\frac{8}{\delta_p}\delta_jq(x_j)
\end{align*}
which is equivalent to
\begin{equation}\label{eq:q-penalized}
\frac{3kq(x_j)^2-\frac18}{(1+q(x_j))^2}
\le \frac{8}{\delta_p}\delta_j=\frac{8}{\delta_p}\cdot\frac1{j(1+r_p(y_j))}
\end{equation}
Using \eqref{v-to-sup_Mv}, one has  
\begin{align}\label{qx_j-tosup}
q(x_j)\longrightarrow \sup_{M^4}q
\end{align}
with the usual convention that \(q_j\to\infty\) when \(\sup_{M^4}q=\infty\).

If \(\sup_{M^4}q=\infty\), then the left-hand side of
\eqref{eq:q-penalized} tends to \(3k\), whereas the right-hand side tends to zero which is  impossible. Hence, \(\sup_{M^4} q<\infty\).

Letting \(j\to\infty\) in \eqref{eq:q-penalized} and using \eqref{qx_j-tosup}, we obtain
\begin{align*}
\frac{3k(\sup_{M^4} q)^2-\frac18}{(1+\sup_{M^4} q)^2}
\le0.
\end{align*}
Therefore, one has 
\begin{align*}
(\sup_{M^4} q)^2\le\frac1{24k}.
\end{align*}
which gives \eqref{Rggeq12k}.
 Thus, we finish the proof.
\end{proof}

\section{Proof of Theorems \ref{thm:Bonnet myers} and \ref{thm:diam}}\label{sec:diameter-control}

\begin{proof}[\bf Proof of Theorem \ref{thm:Bonnet myers}:]
Using  Lemma \ref{lem:GM-lemma},  we have either 
\begin{align}
\Ric_g\equiv -3kg \quad \mathrm{or} \quad R_g>-12k.
\end{align} 
  If $R_g>-12k$, then we obtain $R_g\ge 12k$ by Proposition \ref{prop:gap-R}. 
  Further, Theorem \ref{thm:Bonnet myers-2} implies that $M^4$ is compact. Thus, we finish the proof.
\end{proof}

\begin{proof}[\bf Proof of Theorem \ref{thm:diam}:]
From the assumption \eqref{condition-of-diam}, one has
$$R_g\geq 0, \qquad Q_g\geq 0.$$
Choosing $k=0$ in  Lemma \ref{lem:GM-lemma},  we have either 
\begin{align*}
\Ric_g\equiv 0 \quad  \mathrm{ or }\quad  R_g>0.
\end{align*}    

If $R_g>0$, Lemma \ref{lem:g_S-complete} shows that $R_gg$ is complete.   Apply Lemma \ref{lem:SY} with
	\begin{align*}
	n=4,
	\quad f=R_g,
	\quad \sigma=\frac12,
	\quad \tilde g=R_gg.
	\end{align*}
Consider a curve $\gamma(s)$ on $[0,l]$ chosen as in Lemma \ref{lem:SY}. 
	Set
    \begin{align*}
        z(s):=\log R_g(\gamma(s)).
    \end{align*}
	Then,  for every smooth  function
	$\phi$ vanishing at end points, using \eqref{eq:conf-Ric-R} with $\varepsilon=0$ and the condition \eqref{condition-of-diam},   one has 
	\begin{equation*}
		3\int_0^l(\phi')^2
		-\frac12\int_0^l z'\phi\phi'
		-\frac5{16}\int_0^l(z')^2\phi^2
		\geq 3\theta\int_0^l \phi^2.
	\end{equation*}
	Using Young's inequality \eqref{Young's ineqaulity} in the above inequality, we obtain that
	\begin{equation}\label{phi'geq3theta}
		\frac{16}{5}\int^l_0(\phi')^2\ud s
	\geq 3\theta \int^l_0\phi^2\ud s.
	\end{equation}
	In  \eqref{phi'geq3theta}, choosing 
    \begin{align*}
\phi(s)=\sin\left(\frac{\pi}{l}s\right),
    \end{align*} a direct computation yields that 
    \begin{align*}
        l\leq \frac{4\pi }{\sqrt{15\theta}}.
    \end{align*}
    Thus, $M^4$ is compact.  Following the same argument as in  Corollary \ref{cor:R_g-lower-bound}, one easily obtain $R_g\geq 24\theta.$
	Thus, we finish the proof.
    \end{proof}

	\section{Further discussions}\label{sec:discussions}

\begin{example}\label{example}
Fix $k>0$, for every integer $j\geq1$, let
$$
\mathcal E_j=
\left\{
(x_0,x)\in\mr\times\mr^4:
\frac{k}{j^2}x_0^2+4kj^2|x|^2=1
\right\},
$$
and let $g_j$ be the metric induced from $\mr^5$. After identifying $\mathcal E_j$ with $\ms^4$, we have
\begin{equation*}
  Q_{g_j}\geq6k^2 \quad \text{and} \quad  R_{g_j}\geq12k, 
\end{equation*}
but
\begin{equation*}
   \diam(\ms^4,g_j)\geq d_{g_j}(N_j, S_j)\ge |N_j-S_j|=\frac{2j}{\sqrt{k}}\longrightarrow\infty,
   \end{equation*}
where $N_j$ and $S_j$ are the two poles. We also have
\begin{equation*}
\inf_{\ms^4}\frac{Q_{g_j}}{R_{g_j}}\longrightarrow0.
\end{equation*}
\end{example}

Although the diameter cannot be controlled in this setting, the volume can still be effectively bounded.
Building on Theorem \ref{thm:Bonnet myers}, if we assume the curvature conditions
\begin{equation}\label{Q-R-positive condition}
Q_g \geq 6k^2 \quad \text{and} \quad R_g \geq 12k > 0,
\end{equation}
then Gursky's formula \eqref{Gursky upper bound} immediately yields the sharp volume upper bound
\begin{equation}\label{volume-upper-bound}
\operatorname{Vol}(M^4,g) \leq \frac{|\mathbb{S}^4|}{k^2}.
\end{equation}
The corresponding rigidity result for \eqref{volume-upper-bound} has been addressed in Theorem 1.6 of \cite{LW}.

If \((M^4,g)\) is compact, then under the pointwise positivity assumptions
\begin{equation}\label{Q-and-R-positive-everywhere}
    Q_g > 0, \quad \text{and} \quad R_g > 0, \quad \forall x \in M^4,
\end{equation}
the  result of Chang-Gursky-Yang \cite{CGY} ensures the existence of a conformal metric with positive Ricci curvature. This naturally raises the question of whether \eqref{Q-and-R-positive-everywhere} alone can directly imply that the Ricci curvature of the original metric is positive. Such a question seems rather ambitious. In \cite{LWX}, for conformal metrics \(g = e^{2u}g_0\) on the standard four-sphere \((\mathbb{S}^4, g_0)\), it was shown that if both the \(Q\)-curvature and the scalar curvature are nonnegative everywhere, then the Ricci curvature itself is nonnegative. Moreover, in \cite{LX}, under the same assumptions, they further established that even the sectional curvature is nonnegative.


\begin{thebibliography}{99}

\bibitem{Ambrose}
W. Ambrose,  A theorem of Myers, Duke Math. J.  24 (1957), 345–348.

\bibitem{Besse}
A. Besse, Einstein manifolds, 
Ergeb. Math. Grenzgeb. (3), 10
Springer-Verlag, Berlin, 1987. xii+510 pp.

\bibitem{BHJ}
S. Brendle,  S. Hirsch,  F. Johne, 
A generalization of Geroch's conjecture, 
Comm. Pure Appl. Math. 77 (2024), no. 1, 441–456.

\bibitem{Calabi}
E. Calabi,  On Ricci curvature and geodesics,
Duke Math. J. 34 (1967), 667–676.

\bibitem{CGY}
S.-Y. A.   Chang, M. Gursky,   P. Yang,
An equation of Monge-Ampère type in conformal geometry, and four-manifolds of positive Ricci curvature,
Ann. of Math. (2) 155 (2002), no. 3, 709–787.
\bibitem{CGY JAM}
S.-Y. A.   Chang, M. Gursky,   P. Yang,
An a priori estimate for a fully nonlinear equation on four-manifolds,
J. Anal. Math. 87 (2002), 151–186.


\bibitem{CGY IHES}
S.-Y. A.   Chang, M. Gursky,   P. Yang, A conformally invariant sphere theorem in four dimensions,
Publ. Math. Inst. Hautes Études Sci. No. 98 (2003), 105–143.
\bibitem{CHY}
S.-Y. A.  Chang, F. Hang, P. Yang, On a class of locally conformally flat manifolds, Int. Math. Res. Not. (2004), no. 4, 185–209.
\bibitem{Chang-Yang}
S.-Y. A.  Chang, P.  Yang,
Extremal metrics of zeta function determinants on 4-manifolds,
Ann. of Math. (2) 142 (1995), no. 1, 171–212.
\bibitem{CGZ}
S.-Y. A.  Chang, M.  Gursky, S.  Zhang, 
A conformally invariant gap theorem characterizing $\mathbb{CP}^2$ via the Ricci flow, 
Math. Z. 294 (2020), no. 1-2, 721–746.


\bibitem{CGT}
J. Cheeger, M. Gromov, M.  Taylor,  Finite propagation speed, kernel estimates for functions of the Laplace operator, and the geometry of complete Riemannian manifolds, 
J. Differential Geometry 17 (1982), no. 1, 15–53.




\bibitem{Cheng}
S. Cheng, Eigenvalue comparison theorems and its geometric applications, Math. Z. 143 (1975), no. 3, 289–297.

\bibitem{DM}
Z. Djadli, A.  Malchiodi, 
Existence of conformal metrics with constant $Q$-curvature,
Ann. of Math. (2) 168 (2008), no. 3, 813–858.


\bibitem{GWW}
Y. Ge, G.  Wang, W.  Wei,  A Yamabe problem for the quotient between the Q curvature and the scalar curvature, arXiv:2603.15074.




\bibitem{Gursky-IUMJ}
M. 	Gursky, Locally conformally flat four- and six-manifolds of positive scalar curvature and positive Euler characteristic,
Indiana Univ. Math. J. 43 (1994), no. 3, 747–774.

\bibitem{Gursky Ann}
M. Gursky,  The Weyl functional, de Rham cohomology, and Kähler-Einstein metrics, 
Ann. of Math. (2) 148 (1998), no. 1, 315–337.

\bibitem{GUrsky CMP 97}
M. Gursky, 
Uniqueness of the functional determinant,
Comm. Math. Phys. 189 (1997), no. 3, 655–665.
\bibitem{Gursky CMP}
M. Gursky, The principal eigenvalue of a conformally invariant differential operator, with an application to semilinear elliptic PDE,  Comm. Math. Phys. 207 (1999), no. 1, 131–143.
\bibitem{GM}
M. Gursky, A.  Malchiodi, A strong maximum principle for the Paneitz operator and a non-local flow for the Q-curvature,
J. Eur. Math. Soc. (JEMS) 17 (2015), no. 9, 2137–2173.
\bibitem{GV}
M. Gursky, J.  Viaclovsky, 
A fully nonlinear equation on four-manifolds with positive scalar curvature, 
J. Differential Geom. 63 (2003), no. 1, 131–154.

\bibitem{Jost}
J. Jost, Riemannian geometry and geometric analysis, Seventh edition, Universitext, Springer, Cham, 2017, xiv+697 pp. ISBN: 978-3-319-61859-3; 978-3-319-61860-9.


\bibitem{KL}
K. Kuwae, X.-D.  Li, New Laplacian comparison theorem and its applications to diffusion processes on Riemannian manifolds,
Bull. Lond. Math. Soc. 54 (2022), no. 2, 404–427.
\bibitem{Lee}
S. Lee, 
Some uniformization problems for a fourth order conformal curvature, J. Funct. Anal. 288 (2025), no. 5, Paper No. 110791, 34 pp.
\bibitem{Li 26 RMI}
M. Li,  Obstructions to prescribed Q-curvature of complete conformal metrics on $\mathbb{R}^n$, 
Rev. Mat. Iberoam. 42 (2026), no. 1, 75–94.

\bibitem{LW}
M. Li, J. Wei, Higher order Bol's inequality and its applications, arXiv:2308.11388v2.
\bibitem{LWX}
M. Li, J. Wei,  X. Xu, On geometry of $Q_g^{(2k)}$-curvature, arXiv:2506.20165.

\bibitem{LX}
M. Li, X. Xu, A sharp isoperimetric inequality and the top order $Q$-curvature, 	arXiv:2607.06951.

\bibitem{X. Li}
X.-M. Li,
On extensions of Myers' theorem, 
Bull. London Math. Soc. 27 (1995), no. 4, 392–396.

\bibitem{Myers}
S. Myers,  Riemannian manifolds with positive mean curvature.
Duke Math. J. 8 (1941), 401–404.

\bibitem{Pin}
Y. Pinchover, 
Topics in the theory of positive solutions of second-order elliptic and parabolic partial differential equations, Spectral theory and mathematical physics: a Festschrift in honor of Barry Simon's 60th birthday, 329–355.
Proc. Sympos. Pure Math., 76, Part 1
American Mathematical Society, Providence, RI, 2007.


\bibitem{ShenYe Duke}
Y. Shen, R. Ye, On stable minimal surfaces in manifolds of positive bi-Ricci curvatures, 
Duke Math. J. 85 (1996), no. 1, 109–116.
\bibitem{ShenYe}
Y. Shen,  R. Ye,
On the geometry and topology of manifolds of positive bi-Ricci curvature, arXiv:dg-ga/9708014.

\bibitem{Simon}
B. Simon, Schrödinger semigroups,  Bull. Amer. Math. Soc. (N.S.) 7 (1982), no. 3, 447–526.



	\end{thebibliography}
\end{document}